\documentclass[10pt,a4paper,onecolumn]{article} 

\usepackage[pdftex]{graphicx}
\usepackage{amssymb}
\usepackage{amsmath}
\usepackage{amsthm}
\usepackage{multirow}
\usepackage[square,comma,numbers,sort&compress]{natbib}
\usepackage{url}
\usepackage[font={small}, labelfont=bf]{caption} 

\newcommand{\f}{\operatorname}

\theoremstyle{plain}
\newtheorem{theorem}{Theorem}[section]

\newtheorem{proposition}[theorem]{Proposition}

\theoremstyle{definition}
\newtheorem{definition}[theorem]{Definition}

\theoremstyle{remark}

\usepackage{hyperref}

\usepackage{xcolor} 
\usepackage{environ}         
\usepackage{etoolbox}        
\usepackage{graphicx}        
\newlength{\myl}
\let\origequation=\equation
\let\origendequation=\endequation
\RenewEnviron{equation}{
  \settowidth{\myl}{$\BODY$}  
  \origequation
  \ifdimcomp{\the\linewidth}{>}{\the\myl}
  {\ensuremath{\BODY}}                             
  {\resizebox{\linewidth}{!}{\ensuremath{\BODY}}}  
  \origendequation
}

\title{Objective Bayesian analysis for the differential entropy of the Gamma distribution}

\author{
Eduardo Ramos$^1$, Osafu A. Egbon$^1$, Pedro L. Ramos$^2$, \\ Francisco A. Rodrigues$^1$  and Francisco Louzada$^1$ \\
\normalsize{$^{1}$Institute of Mathematical Science and Computing, University of S\~ao Paulo, S\~ao Carlos, Brazil} \\
\normalsize{$^{2}$Facultad de Matemáticas, Pontificia Universidad Católica de Chile, Macul, Santiago 7820436, Chile}
}

\begin{document}

\maketitle

\begin{abstract}
 The present paper introduces a fully objective Bayesian analysis to obtain the posterior distribution of an entropy measure. Notably, we consider the gamma distribution, which describes many natural phenomena in physics, engineering, and biology. We reparametrize the model in terms of entropy, and different objective priors are derived, such as Jeffreys prior, reference prior, and matching priors. Since the obtained priors are improper, we prove that the obtained posterior distributions are proper and that their respective posterior means are finite. An intensive simulation study is conducted to select the prior that returns better results regarding bias, mean square error, and coverage probabilities. The proposed approach is illustrated in two datasets: the first relates to the Achaemenid dynasty reign period, and the second describes the time to failure of an electronic component in a sugarcane harvest machine.
\end{abstract}

\section{Introduction} \label{sec:introduction}

In recent years, there has been a growing interest in estimating different metrics of information theory related to parametric distributions. The Shannon entropy, also known as differential entropy, introduced by Claude  Shannon \citep{shannon1948mathematical}, is an essential quantity that measures the amount of available information or uncertainty outcome of a random process. Given a density function $f(x|\alpha,\beta)$, the differential entropy is given by 
\begin{equation}
    H(\alpha,\beta)=\mathbb{E}\left(-\log f(x|\alpha,\beta)\right).
\end{equation}
{\color{black}Shannon entropy is a versatile tool that can be used to analyze various systems. In applied statistics, for example, it can be utilized to measure the level of disorder within a mechanical or biological system. In this context, a higher entropy value associated with a random phenomenon may indicate the presence of irregularities in the system and dispersion from homogeneity. It can aid in the examination of species diversity and heterogeneity within an ecosystem \citep{jost2006entropy}.  
As will be shown later in this work, the Shannon entropy can be used to study the instability of political institutions, specifically in the Achaemenid dynasty. }

The Shannon entropy {\color{black}of a probability distribution is a function of the }distribution parameters that must be estimated from a sample. A commonly used method to estimate the parameters is the maximum likelihood approach due to its one-to-one invariance property. Hence, {\color{black}to estimate the entropy,} it is only required to estimate the parameters of the original model and plug them into the entropy function. Under this approach many authors have derived the entropy estimators for different distributions such as Weibull \citep{cho2015estimating}, Inverse Weibull \citep{yu2019statistical}, Log-logistic \citep{du2018statistical} and for the exponential distribution with different shift origin \citep{kayal2013estimation}, to list a few.

A major drawback of the maximum likelihood inference is that the obtained estimates are usually biased for small samples \citep{cordeiro2014introduction}. Another concern under small samples happens when constructing the confidence intervals for the parameters since such intervals are not precise and may not return good coverage probabilities.  {\color{black} To overcome these limitations, objective Bayesian methods could be adopted as they usually return more precise estimates in terms of less bias \citep{firth1993bias} and coverage probabilities \citep{tibshirani1989}. Several objective prior distributions have been developed for the Gamma distribution. For example, the work of }
\cite{miller1980bayesian}, \cite{sun1996frequentist}, \cite{berger2015}, and  \cite{louzada2018efficient}. More recently, \cite{Pedro2020onposterior} revised the most common objective priors and provided sufficient and necessary conditions for the obtained posteriors and their higher moments to be proper. {\color{black}However, there is a significant lack of adequate attention in the objective Bayesian analysis of gamma entropy, which is a resourceful tool for studying diverse systems.}

Although {\color{black}previous works have} obtained different joint posterior distributions for the parameters of interest, the obtained posterior {\color{black}quantities} can not be directly plunged in the Shannon entropy {\color{black}function, unlike the MLE case. Hence it becomes imperative to develop a full Bayesian analysis of the Shannon entropy itself}. {\color{black} That is, } under the Bayesian approach, it is necessary to obtain the posterior distribution of the entropy measure. 
As a result, \cite{shakhatreh2020objective} derived different posterior distributions using objective priors for the entropy assuming a Weibull distribution. {\color{black}However, studies on the posterior distributions of the entropy of gamma distribution using objective priors have not been considered. Unlike the Weibull distribution, the complex nature of the expression of gamma entropy makes the objective Bayesian analysis challenging.}
The gamma distribution considered here is a two-parameter family of distributions, which is among the most well-known distributions used to model different stochastic processes and to make statistical inferences. It has received attention from different fields. It has surfaced in many areas of applications, including financial analysis~\citep{cizek2005statistical}, climate analysis~\citep{husak2007use}, reliability analysis~\citep{gupta1961gamma}, machine learning~\citep{kamalov2020gamma}, and physics~\citep{garcia2012stochastic}.
Particularly, the gamma distribution includes the exponential distribution, Erlang distribution, and chi-square distribution as special cases. 

A random variable $X$ follows a gamma distribution, if its probability density function, parametrized by a shape parameter $\alpha>0$ and scale parameter $\beta>0 $, is given by,
\begin{equation}\label{pdfgammma}
f(x\,| \alpha,\beta)=\frac{\beta^\alpha}{\Gamma(\alpha)}x^{\alpha-1}e^{-\beta x} ,~~ x>0,
\end{equation}
where $\Gamma(\phi)=\int_{0}^{\infty}{e^{-x}x^{\phi-1}dx}$ is the gamma function. 

\textcolor{black}{In this paper, focusing on the gamma distribution, we derive the posterior distributions using objective priors, such as Jeffreys prior \cite{jeffreys1946invariant}, reference priors \citep{bernardo1979a, berger1992development,berger2015}, and matching priors \citep{tibshirani1989}, and prove that the obtained posteriors are proper and can be used to construct the posterior distributions of the Shannon entropy. Although the posterior distribution may be proper, the posterior mean can be infinite, which is undesirable. Therefore, we further proved that the obtained posterior means for the gamma entropy measure are finite. Credibility intervals are obtained to construct accurate interval estimates. Since the posterior distributions of $H(\cdot)$ do not have a closed form, we considered a simple Metropolis-Hastings algorithm to sample from the posterior, which is efficient and fast for our proposed model. A function in R software is presented to estimate such results automatically by setting good initial values and returning the posterior estimates. Finally, we applied our proposed results to estimate the entropy related to the failure times in sugarcane harvest machines and in the Achaemenid dynasty's rule time to quantify the variability in the Persian Empire's political institutions.}

The paper is organized as follows. Section \ref{entropysec} presents the maximum likelihood estimators for the gamma distribution parameters and the Shannon Entropy computation. Section \ref{bayesianinf} presents the objective Bayesian analysis using objective priors for the Shannon entropy parameter's reparametrized posterior distribution. Section \ref{simulations} provides a simulation study to select the best objective prior. In Section \ref{application}, the methodology is illustrated on a real dataset. Some final comments are given in Section \ref{conclusions}.

\section{Frequentist approach}\label{entropysec}

The classical inference (frequentist) is a commonly used approach to conduct parameter estimation of a particular distribution. In this case, the parameter is treated as fixed, and the MLE is commonly used to obtain the estimates. The MLE has good asymptotic properties, such as invariance, consistency, and efficiency. This procedure search the parameter space of $\boldsymbol{\theta}$ where the maximum likelihood $\hat{\boldsymbol{\theta}}=\sup_{\boldsymbol{\theta}}L(\boldsymbol{\theta}|x)$ is obtained. Here our main aim is to obtain the estimate of a function of the parameters. Hence, firstly we need to obtain the entropy measure, mathematically defined as $H(\boldsymbol{\theta})=E(-log f(x|\boldsymbol{\theta}))$, which quantifies the amount of uncertainty in the data $x$. 
Besides, it should be noted that a higher realization of $H$ indicates more uncertainty.

The entropy $H$ of the gamma density is given by
\begin{equation}\label{entropy}
\begin{aligned}
H(\alpha,
\beta)&=-\int^\infty_0\log\left(\frac{\beta^\alpha}{\Gamma(\alpha)}x^{\alpha-1}exp\{-\beta
x\}\right)f(x|\alpha,\beta)dx \\ & 
=\alpha-\log(\beta)+\log(\Gamma(\alpha))+(1-\alpha)\psi(\alpha),
\end{aligned}
\end{equation}
where \textcolor{black}{$\psi(k)=\frac{\partial}{\partial k}\log\Gamma(k)$} is the digamma function.

\textcolor{black}{There are several methods to estimate $H(\cdot)$ from a classical standpoint. The first method involves utilizing plug-in estimators, where we first obtain the Maximum Likelihood Estimators (MLEs) $\hat\alpha$ and $\hat\beta$ for the parameters $\alpha$ and $\beta$, respectively. Subsequently, we plug $\hat\alpha$ and $\hat\beta$ into $H(\hat\alpha,\hat\beta)$ to obtain the MLE of $H(\cdot)$, denoted as $\hat{H}(\cdot)$. To construct confidence intervals, we can employ the delta method to approximate the variance of $\hat{H}(\cdot)$. However, as per Oehlert (1992), the delta method is primarily a technique for approximating expected values of functions of random variables when direct evaluation of the expectation is not feasible. The delta method approximates the expectation of $H(\cdot)$ by taking the expectation of a polynomial approximation to $H(\cdot)$. Alternatively, the second method involves considering a change of variable, which could yield more precise results as it does not rely on approximations.}

 Now, consider a change of variable by setting $W=\alpha$, which implies $H = W - \log(\beta) + \log\Gamma(W) + (1-W)\psi(W)$. The aim of the transformation is to obtain a likelihood of $H$ and $W$ instead of $\alpha$ and $\beta$. Therefore, if $X_1$, $\ldots$, $X_n\,$ are a complete sample from (\ref{pdfgammma}) then the likelihood function of $H$ and $W$ is given as
\begin{equation}\label{eq5}
L(W, H\,|
\boldsymbol{x})=\frac{\delta(W,H)^{nW}}{\Gamma(W)^n}\left\{\prod_{i=1}^n{x_i^W}\right\}\exp\left\{-\delta(W,H)\sum_{i=1}^n x_i\right\},
\end{equation}
where $\delta(W,H)=\exp\left(W + \log\Gamma(W) + (1 - W)\psi(W) - H\right)$. 

The log-likelihood function is given by
\begin{equation}\label{eq6}
\begin{aligned}
l(W, H\,|\boldsymbol{x})=W\log(\delta(W,H))-\textcolor{black}{n}\log\Gamma(W) + W\sum_{i=1}^n \log(x_i) - \delta(W,H)\sum_{i=1}^{n}x_i.\\
\end{aligned}
\end{equation}

The MLEs for the parameters are obtained by directly maximizing the log-likelihood function \textcolor{black}{$l(W, H\,|\boldsymbol{x})$}. Hence,
after some algebraic manipulations the MLEs $\hat{W}$ and $\hat{H}$ are obtained from the solution of
\begin{equation*}
\frac{\partial l(W,H\,|\boldsymbol{x})}{\partial W}=\log(\delta(W,H))-\psi(W) - \sum_{i=1}^n \log(x_i) + \sigma \left(W - \delta(W,H)\sum_{i=1}^n x_i\right)
\end{equation*}
\begin{equation*}
\frac{\partial l(W,H\,|\boldsymbol{x})}{\partial H}=-W + \delta(W,H)\sum_{i=1}^{n}x_i
\end{equation*}
 where $\sigma = 1 + (1- W)\psi'(W)$.
The solutions for these equations provide the maximum likelihood estimators for the entropy of the gamma distributions,
$\widehat{H}$ and $\widehat{W}$. Since equation (3) cannot be solved easily using a closed-form solution,
numerical techniques must estimate the true parameters.

Following \cite{migon}, the MLEs are asymptotically normally distributed with a joint bivariate normal distribution given by
\begin{equation*} 
\left(\hat{W}_{\rm MLE},\hat{H}_{\rm MLE}\right) \sim N_2 \left[\left(W,H \right),I^{-1} \left(W,H\right) \right] \quad \mbox{ as } \quad n \to \infty , 
\end{equation*}
where $I(W,H)$ is the Fisher information matrix for the reparametrized model given by
 \begin{align}
 \begin{aligned}
   I(W,H)
=&\begin{bmatrix}
\psi'(W) -2\sigma + W\sigma^2 & 1-\sigma W\\
1-\sigma W & W 
\end{bmatrix},
\end{aligned}
 \end{align}

\vspace{0.2cm}
\noindent and $\psi\,'$($W$) is the derivative of $\psi$($W$), called the
trigamma function.

In the present paper, we are only interested in $H$, and thus, given $0<a<1$ and using the element $ (I(W,H)^{-1})_{22}$, we can conclude that the confidence interval for the estimate of the entropy measure with a  confidence level of $100(1-a)\%$ for $H$ is given by
\begin{equation}
    \hat{H}-{Z_{\frac{a}{2}}}\sqrt{(1 - W)^2\psi'(W) + 2 - W}< H<\hat{H}+{Z_{\frac{a}{2}}}\sqrt{(1 - W)^2\psi'(W) + 2 - W},
\end{equation}
where $a$ is the significance level and ${Z_{\frac{a}{2}}}$ is the $\frac{a}{2}$-th percentile of the standard normal distribution.

\section{Bayesian Inference}\label{bayesianinf}

Here, the parameter $\boldsymbol{\theta}$ is considered as a random variable, and the distribution that represents knowledge about $ \boldsymbol {\theta}$ is refereed as a prior distribution and defined by $\pi(\boldsymbol{\theta})$. The distribution $\pi(\boldsymbol{\theta})$ provides the knowledge or uncertainty about $\boldsymbol{\theta}$ before obtaining the sample data $\boldsymbol{x}$. After the data $x$ is observed, a natural way of combining the resulting information from the a priori distribution and the likelihood function is done by Bayes' theorem, resulting in the posterior distribution of $\boldsymbol{\theta}$ given $\boldsymbol{x}$. 

\textcolor{black}{Assuming a Bayesian approach our specific interest lies in estimating the entropy related to the model, which is denoted as $g(\boldsymbol{\theta})=H(\boldsymbol{\theta})$. As $\boldsymbol{\theta}$ is a random variable, $g(\boldsymbol{\theta})$ becomes a random variable as well. In practical situations, determining the transformation can be challenging, often leading practitioners to rely on plug-in estimators. Despite this, correcting and deriving the distribution of $g(\boldsymbol\theta)$ is sometimes feasible, rendering it a random variable. In such cases, we obtain the posterior distribution of the entropy $H(\cdot)$. In this work, we demonstrate the ability to derive this distribution and directly compute results from these posteriors.}

To obtain the posterior distributions for the $H$ parameter, we can consider the one-to-one invariance property of the Jeffreys prior, reference prior, and matching prior, and thus we only need to obtain the Jacobian matrix related to the reparametrization from $\alpha$ and $\beta$ to $H$ and $W$. After some algebraic manipulations, we can conclude that the parameters $\beta$ and $\alpha$ can be written as
\begin{equation*}
\beta = \exp\left(W + \log(\Gamma(W)) + (1 - W)\psi(W) - H\right) \quad \mbox{ and } \quad \alpha=W,
\end{equation*}
and thus, from the relations 
     \begin{align*}
      \frac{\partial \alpha}{\partial H}=0,\, 
      \frac{\partial \alpha}{\partial W}=1,\,
     \frac{\partial \beta}{\partial H}=- \beta \ \ \mbox{ and } \ \ 
     \frac{\partial \beta}{\partial W}= \left( 1+(1-W)\textcolor{black}{\psi^{'}(W)} \right)\beta,
 \end{align*}  
it follows that the Jacobian matriz (J) relative to the change of variable will be given by
 \begin{align}
   J = \begin{bmatrix}
\frac{\partial \alpha}{\partial H} & \frac{\partial \alpha}{\partial W} \\
\frac{\partial \beta}{\partial H} & \frac{\partial \beta}{\partial W}
\end{bmatrix}
=\begin{bmatrix}
0 & 1\\
-\beta& \sigma \beta
\end{bmatrix},
 \end{align}
where $\sigma = 1 + (1-W)\psi'(W)$.

\vspace{0.1 cm}

The use of objective priors plays an essential role in Bayesian analysis where the data provide the dominant information, and the posterior distribution is not overshadowed by prior information. Such priors allow us to conduct objective Bayesian inference. On the other hand, in most situations, they are not proper prior distributions and may lead to improper posterior, invalidating the analysis since we cannot compute the normalizing constant. Therefore, we need to check if the obtained posterior (and posterior mean) is proper (or finite). The priors for the entropy and its related posterior distributions will be discussed in the next subsections.

Before we derive the priors and posterior distributions, hereafter, we shall always assume that there are at least two distinct data $t_i$, that is, there exists $1\leq i<j \leq n$ such that $t_i\neq t_j$. Additionally, before we proceed, we present below a definition and proposition that will be used to prove that the obtained posteriors are proper. In the following let $\overline{\mathbb{R}} = \mathbb{R}\cup \{-\infty, \infty\}$ denote the \textit{extended real number line} and let $\mathbb{R}^+$ denote the strictly positive real numbers. The following definition is a special case from the one presented in \cite{ramos2023power} and will play an important role in proving that the analyzed posterior distributions and posterior means are proper.

\begin{definition}\label{def31}
Let $a\in \mathbb{\overline{R}}$, $\f{g}:\mathcal{U}\to\mathbb{R^+}$ and $\f{h}:\mathcal{U}\to\mathbb{R^+}$, where $\mathcal{U}\subset\mathbb{R}$ and suppose that $\lim_{x\to a} \dfrac{\f{g}(x)}{\f{h}(x)} = c\in \mathbb{R}$. Then, if $c>0$, we say that $g(x)\underset{x\to a}{\propto} h(x)$.
\end{definition}

Regarding the above definition, we have the following proposition from \cite{ramos2017bayesian}.

\begin{proposition}\label{prop32} Let $\f{g}:(a,b)\to\mathbb{R^+}$ and $\f{h}:(a,b)\to\mathbb{R^+}$ be continuous functions in $(a,b)\subset\mathbb{R}$, where $a\in\overline{\mathbb{R}}$ and $b\in\overline{\mathbb{R}}$, and let $c\in(a,b)$. Then $\f{g}(x)\underset{x\to a}{\propto} \f{h}(x)$ implies in $
\int_a^c g(t)\; dt \propto \int_a^c h(t)\; dt$ and  $\f{g}(x)\underset{x\to b}{\propto} \f{h}(x)$ implies in $\int_c^b g(t)\; dt \propto \int_c^b h(t)\; dt$.
\end{proposition} 

\subsection{Jeffreys prior}
\cite{jeffreys1946invariant} described a procedure to achieve an objective prior, which is invariant under one-to-one monotone transformations. The invariant property of the Jeffreys prior has been widely exploited to make statistical inferences from its posterior distribution numerical analysis. The prior construction is based on the square root of the determinant of the Fisher information matrix  $I(\alpha,\beta)$. Thus, the Jeffreys prior to the gamma distribution is given by
\begin{equation*}
\pi_1\left(\alpha,\beta\right)\propto \frac{\sqrt{\alpha\psi'(\alpha)-1}}{\beta}.
\end{equation*}

Additionally, from the determinant of the Fisher information, or using the change of variables over the Jeffreys prior we have
\begin{equation}\label{priorjnk}
\pi_1\left(H,W\right)\propto \sqrt{W\psi'(W)-1}.
\end{equation}
Finally, the joint posterior distribution for $H$ and $W$ produced by the Jeffreys prior is 
\begin{equation}\label{postjnk1} 
\begin{aligned}
\pi_1(H,W|\boldsymbol{x})\propto\frac{\delta(W,H)^{nW}\sqrt{W\psi'(W)-1}}{\Gamma(W)^n}\left\{\prod_{i=1}^n{x_i^W}\right\}\exp\left\{-\delta(W,H)\sum_{i=1}^n x_i\right\} .
\end{aligned}
\end{equation}

\begin{theorem}\label{theo33} The posterior density (\ref{postjnk1}) is proper for all $n\geq 2$.
\end{theorem}
\begin{proof} The proof can be seen in Appendix \ref{ccheo33}.
\end{proof}

\begin{theorem}\label{theo34} The posterior mean of $H$ relative to (\ref{postjnk1}) is finite for any $n\geq 2$.
\end{theorem}

\begin{proof} The proof can be seen in Appendix \ref{ctheo34}.
\end{proof}

To sample for the posterior distribution we obtain that the marginal posterior distribution of $W$ is given by
\begin{equation*}
\pi_1(W|\boldsymbol{x})\propto\sqrt{W\psi'(W)-1}\frac{\Gamma(nW)}{\Gamma(W)^n} \left(\frac{\sqrt[n]{\prod_{i=1}^n{x_i}}}{ \sum_{i=1}^n x_i}\right)^{nW},
\end{equation*}
and the conditional posterior distribution of $H$ is given by
\begin{equation*}
\begin{aligned}
\pi_1(H|W,\boldsymbol{x})\propto\exp\left\{-nWH-\delta(W,H)\sum_{i=1}^n x_i\right\}.
\end{aligned}
\end{equation*}

\subsection{Reference prior}
\cite{bernardo1979a} discussed a different approach to obtain a new class of objective priors, named as reference priors. Further, many studies were presented to develop formal and rigorous definitions to derive such class of prior distributions under different contexts \citep{berger1989estimating, berger1992ordered, berger1992reference, berger1992development, berger2015}. The reference prior is obtained by maximizing the Kullback-Leibler (KL) divergence assuming some regularity conditions. The idea of the expected posterior information to the prior allows the data to have the maximum influence on the posterior distributions. The reference priors have essential properties such as consistent sampling, consistent marginalization, and one-to-one transformation invariance \citep{bernardo2005}. The reference priors may depend on the order of the parameters of interest. Hence, for the gamma distribution, we have two distinct priors that are presented below.

\subsubsection{Reference prior when $\beta$ is the parameter of interest}

The reference prior when $\beta$ is the parameter of interest and $\alpha$ is the nuisance parameter is given by
\begin{equation*}
\pi_2\left(\alpha,\beta\right)\propto \frac{\sqrt{\psi'(\alpha)}}{\beta}.
\end{equation*}
Thus, using the Jacobian transformation it follows that the related reference prior is given by
\begin{equation}\label{priorgmr2f}
\pi_2\left(W,H\right)\propto \sqrt{\psi'(W)}.
\end{equation}
Finally, the joint posterior distribution for $H$ and $W$, produced by the reference prior (\ref{priorgmr2f}), is given by
\begin{equation}\label{postgmr21}
\pi_2(W,H|\boldsymbol{x})\propto\delta(W,H)^{nH}\frac{\sqrt{\psi'(W)}}{\Gamma(W)^n}\left\{\prod_{i=1}^n{x_i^H}\right\}\exp\left\{-\delta(W,H)\sum_{i=1}^n x_i\right\}.
\end{equation}

\begin{theorem}\label{theo35} The posterior density (\ref{postgmr21}) is proper for all $n\geq 2$.
\end{theorem} 
\begin{proof} The proof can be seen in Appendix \ref{ctheo35}.
\end{proof}

\begin{theorem}\label{theo36} The posterior mean of $H$ relative to (\ref{postgmr21}) is finite for all $n\geq 2$.
\end{theorem}
\begin{proof} The proof can be seen in Appendix \ref{ctheo36}.
\end{proof}

The marginal posterior distribution of $W$ is given by
\begin{equation*}
\pi_2(W|\boldsymbol{x})\propto\sqrt{\psi'(W)}\frac{\Gamma(nW)}{\Gamma(W)^n} \left(\frac{\sqrt[n]{\prod_{i=1}^n{x_i}}}{ \sum_{i=1}^n x_i}\right)^{nW}.
\end{equation*}
Moreover, the conditional posterior distribution of $H$ is given by
\begin{equation*}
\begin{aligned}
\pi_2(H|W,\boldsymbol{x})\propto\exp\left\{-nWH-\delta(W,H)\sum_{i=1}^n x_i\right\}.
\end{aligned}
\end{equation*}

\subsubsection{Reference prior when $\alpha$ is the parameter of interest}

The reference prior when $\alpha$ is the parameter of interest and $\beta$ is the nuisance parameter is given by
\begin{equation*}
\pi_3\left(\alpha,\beta\right)\propto \frac{1}{\beta}\sqrt{\frac{\alpha\psi'(\alpha)-1}{\alpha}}.
\end{equation*}
Therefore, in terms of the reparametrized model, the reference prior when $W$ is the parameter of interest and $H$ is the nuisance parameter is given by
\begin{equation}\label{priorgmr1f}
\pi_3\left(W,H\right)\propto \sqrt{\frac{W\psi'(W)-1}{W}}.
\end{equation}
Finally, the joint posterior distribution for $H$ and $W$, produced by the reference prior (\ref{priorgmr1f}) is given by
\begin{equation}\label{postref1}
\pi_3(W,H|\boldsymbol{x})\propto\sqrt{\frac{W\psi'(W)-1}{W}}\frac{\delta(W,H)^{nW}}{\Gamma(W)^n}\left\{\prod_{i=1}^n{x_i^W}\right\}\exp\left\{-\delta(W,H)\sum_{i=1}^n x_i\right\}.
\end{equation}

\begin{theorem}\label{theo37} The posterior density (\ref{postref1}) is proper for all $n\geq 2$.
\end{theorem} 
\begin{proof} Since $\sqrt{\frac{W\psi'(W)-1}{W}}=\sqrt{\psi'(W)-\frac{1}{W}}\leq \sqrt{\psi'(W)}$ it follows that $\pi_3(W,H)\leq \pi_2(W,H)$ for all $W\in (0,\infty)$ and $H\in (-\infty,\infty)$ and thus Theorem \ref{theo37} follows directly from Theorem \ref{theo35}.
\end{proof}

\begin{theorem}\label{theo38} The posterior mean of $H$ relative to (\ref{postref1}) is finite for all $n\geq 2$.
\end{theorem}
\begin{proof} Since $\sqrt{\frac{W\psi'(W)-1}{W}}=\sqrt{\psi'(W)-\frac{1}{W}}\leq \sqrt{\psi'(W)}$, it follows that $|H\pi_3(W,H)|=|H|\pi_3(W,H)\leq |H|\pi_2(W,H)= |H\pi_2(W,H)|$ for all $W\in (0,\infty)$ and $H\in (-\infty,\infty)$ and thus Theorem \ref{theo38} follows directly from Theorem \ref{theo36}
\end{proof}

The marginal posterior distribution of $W$ is given by
\begin{equation*}
\pi_3(W|\boldsymbol{x})\propto\sqrt{\frac{W\psi'(W)-1}{W}}\frac{\Gamma(nW)}{\Gamma(W)^n} \left(\frac{\sqrt[n]{\prod_{i=1}^n{x_i}}}{ \sum_{i=1}^n x_i}\right)^{nW}.
\end{equation*}
Moreover, the conditional posterior distribution of $H$ is given by
\begin{equation*}
\begin{aligned}
\pi_3(H|W,\boldsymbol{x})\propto\exp\left\{-nWH-\delta(W,H)\sum_{i=1}^n x_i\right\}.
\end{aligned}
\end{equation*}

\subsection{Matching priors}
\cite{tibshirani1989} considered a different method to obtain a class of one parameter non-informative prior distribution with nuisance parameters. Letting $\pi(\theta_1, \theta_2)$ be a prior distribution with the parameter of interest $\theta_1$ and a nuisance $\theta_2$, the proposed approach requires that the resulting credible interval of the posterior distribution for $\theta_1$ have a frequentist coverage accurate to $O$($n^{-1}$), that is, it requires that
\begin{equation}\label{matchingp}
P\left[\theta_1\leq\theta_1^{1-\alpha}(\pi;X)|(\theta_1,\theta_2)\right]=1-\alpha-O(n^{-1}),
\end{equation}
where $\theta_1^{1-\alpha}(\pi;X)|(\theta_1,\theta_2)$ denotes the $(1-\alpha)$th quantile of the posterior distribution of $\theta_1$. The priors that satisfy (\ref{matchingp}) up to $O(n^{-1})$ are know as matching priors. Under parametric orthogonality, \cite{mukerjee1993frequentist} discussed sufficient and necessary conditions for a class of Tibshirani priors to be a matching prior up to $o(n^{-1})$.
\cite{sun1996frequentist}  derived a   forward and backward reference prior  \citep{berger1989estimating} for a two-parameter exponential family, and further showed that the reference prior are special cases of the matching priors. For a gamma distribution, they showed that the reference prior (\ref{priorgmr1f}) is a matching prior when $\beta$ is set as a nuisance parameter and $\alpha$ is the interest parameter, and proved that there exist no matching prior up to order $O(n^{-1})$. Again, they showed that the reference prior (\ref{priorgmr1f}) is a matching prior when $\beta$ is the interest parameter and $\alpha$ is the nuisance parameter with order $O(n^{-1})$ and proved that there exists a matching prior up to order $o(n^{-1})$. The cited matching  prior is defined as
\begin{equation}\label{priorgmtb1f}
\pi_{4}\left(\alpha,\beta\right)\propto \frac{\alpha\psi'(\alpha)-1}{\beta\sqrt{\alpha}}.
\end{equation}
Thus, the reparametrized version of the proposed matching prior is given by
\begin{equation}\label{priorgmr2f}
\pi_4\left(H,\delta(W,H)\right)\propto \frac{W\psi'(W)-1}{\sqrt{W}}.
\end{equation}
Finally, the joint posterior distribution for $H$ and $W$, produced by the matching prior (\ref{priorgmr2f}) is given by
\begin{equation}\label{postgmr22}
\pi_4(W,H|\boldsymbol{x})\propto\delta(W,H)^{nH}\frac{(W\psi'(W)-1)}{\sqrt{W}\, \Gamma(W)^n}\left\{\prod_{i=1}^n{x_i^H}\right\}\exp\left\{-\delta(W,H)\sum_{i=1}^n x_i\right\}.
\end{equation}

\begin{theorem}\label{theo39} The posterior density (\ref{postgmr22}) is proper for all $n\geq 2$.
\end{theorem} 
\begin{proof} The proof can be seen in Appendix \ref{ctheo39}.
\end{proof}

\begin{theorem}\label{theo310} The posterior mean of $H$ relative to (\ref{postgmr22}) is finite for all $n\geq 2$.
\end{theorem} 
\begin{proof} The proof can be seen in Appendix \ref{ctheo310}.
\end{proof}

The marginal posterior distribution of $W$ is given by
\begin{equation*}
\pi_4(W|\boldsymbol{x})\propto\frac{W\psi'(W)-1}{W^{\frac{1}{2}}}\frac{\Gamma(nW)}{\Gamma(W)^n} \left(\frac{\sqrt[n]{\prod_{i=1}^n{x_i}}}{ \sum_{i=1}^n x_i}\right)^{nW}.
\end{equation*}
Moreover, the conditional posterior distribution of $H$ is given by
\begin{equation*}
\begin{aligned}
\pi_4(H|W,\boldsymbol{x})\propto\exp\left\{-nWH-\delta(W,H)\sum_{i=1}^n x_i\right\}.
\end{aligned}
\end{equation*}

\section{Simulation Study}\label{simulations}

A Monte Carlo simulation study is conducted to quantify and compare the different non-informative priors' impact on the entropy measure's posterior distribution. The Bias and Mean Square Error (MSE) were used to identify the prior that provides the posterior distribution with posterior estimates closer to the true value. These metrics are given by 
\begin{equation}
\f{Bias}_{H}=\frac{1}{N}\sum_{i=1}^{N}(\hat{H}_{i}-H) \ \ \mbox{ and } \ \ \f{MSE}_H\sum_{i=1}^{N}\frac{(\hat{H}_{i}-H)^2}{N},
\end{equation} 
where $N=10,000$ is the number of samples used to estimate the MLE and posterior quantities of interest. Here, we used the posterior mean as the Bayes estimate due to its good properties. The estimates of $W$ are not presented since we only considered $W$ as an auxiliary parameter to conduct the Jacobian transformation, and therefore we are not interested in its respective estimates. 
 
In addition to the Bias and MSE, the coverage probabilities $CP$ were also presented. Such metrics were obtained from the Bayesian credibility intervals (CI) and the asymptotic confidence intervals of $H$. The nominal level assumed was 0.95, i.e., we expect an adequate procedure to compute the confidence/credibility intervals should return coverage probabilities closer to 0.95. Regarding the Bias and MSE, the best approach among the selected ones should return the Bias and MSE closest to zero. 

The Newton-Raphson iterative method was used to maximize the likelihood in order to obtain the MLE. For a fair comparison, the initial values used to start the iterative procedures were the same values as those used to generate the samples. \textcolor{black}{In real applications, there is a need to set initial values. To this end, we can use the closed-form maximum a posteriori estimator derived by \cite{louzada2018efficient} given by \begin{equation}\label{ini01}  
\tilde\alpha=\left(\frac{n-2.9}{n}\right)\cfrac{n\sum_{i=1}^n t_i}{\left(n\sum_{i=1}^{n}t_i \log\left(t_i\right) - \sum_{i=1}^n t_i \sum_{i=1}^n \log\left(t_i\right) \right)} 
\end{equation}
and
\begin{equation}\label{ini02}
\tilde\beta=\frac{1}{n^2}\left(n\sum_{i=1}^{n}t_i \log\left(t_i\right) - \sum_{i=1}^n t_i \sum_{i=1}^n \log\left(t_i\right) \right).
\end{equation}
Therefore, the initial values for $H$ and $W$ are computed from  $\tilde{H}=\tilde{\alpha}-\log(\tilde{\beta})+\log\Gamma(\tilde{\alpha})+(1-\tilde{\alpha})\psi(\tilde{\alpha})$ and $\tilde{W}=\tilde\alpha$.}

In the Bayesian framework, the posterior distribution's marginal densities involve double integrals to obtain the normalizing constants. Therefore, the MCMC approach was adopted to obtain the posterior estimates. Moreover, the Metropolis-Hastings algorithm was adopted to simulate quantities of interest from the posterior densities. The first 500 samples were discarded in the burn-in stage for each dataset simulated, and 5,000 iterations were further conducted. It was considered a thinning parameter of 5 to avoid significant autocorrelation among the samples, returning at the end 1,000 simulated values for each marginal distribution. The Geweke diagnostics \cite{geweke1992evaluating} was considered to confirm the convergence of chains under a confidence level of 95\%. \textcolor{black}{The effective sample size was also adequate for the chains.} The generated samples were used to estimate the posterior mean and the credibility intervals, resulting in 10,000 estimates for $H$ and $W$.

The R software (R Core Team) was used for the simulation, where the codes can be obtained upon request from the corresponding author. For $n=(20,\ldots,120)$, only the result sets at \textcolor{black}{$(\alpha,\beta)=(1,3)$}, $(\alpha,\beta)=(4,2)$ and $(\alpha,\beta)=(2,0.5)$ were presented, which led respectively to \textcolor{black}{$H=-0.01$}, $H=1.33$ and $H=2.27$. However, the results were similar for different values of $\alpha$ and $\beta$ and therefore were not presented here. For each sample from the posterior distribution, the posterior mode and the credible intervals were evaluated for $\alpha$, $\beta$, and $H$.

\begin{table}[!h]
\centering
{\scriptsize
\caption{The Bias(MSE) from the estimates of $\mu$ considering different values of $n$  with $N=10,000$ simulated samples, using the estimation methods: 1 - MLE, 2 - Jeffreys's rule, 3 -  Reference 1  prior, 4 - Reference  2 prior, and 5 - Tibshirani Prior.}
\begin{tabular}{c|c|c|c|c|c|c}
\hline
$H$ & n  & MLE & Jeffreys's & Reference 1 & Reference 2 & Tibshirani \\
\hline
& 20  & 0.0516(0.0550) & 0.0504(0.0549) & 0.0514(0.0550) & 0.0514(0.0551) & 0.0541(0.0554) \\
& 30  &  0.0337(0.0361) & 0.0334(0.0361) & 0.0340(0.0361) & 0.0340(0.0362) & 0.0355(0.0362) \\
& 40  &  0.0232(0.0258) & 0.0231(0.0258) & 0.0234(0.0258) & 0.0233(0.0259) & 0.0241(0.0260) \\
& 50  &  0.0209(0.0210) & 0.0209(0.0210) & 0.0210(0.0210) & 0.0211(0.0210) & 0.0216(0.0211) \\
& 60  &  0.0165(0.0171) & 0.0165(0.0171) & 0.0167(0.0171) & 0.0166(0.0171) & 0.0169(0.0171) \\
& 70  &  0.0142(0.0149) & 0.0141(0.0149) & 0.0143(0.0150) & 0.0143(0.0150) & 0.0146(0.0150) \\
$-0.01$    & 80  &  0.0145(0.0130) & 0.0145(0.0130) & 0.0146(0.0130) & 0.0146(0.0130) & 0.0148(0.0130) \\
& 90  &  0.0092(0.0113) & 0.0092(0.0113) & 0.0092(0.0113) & 0.0093(0.0113) & 0.0094(0.0113) \\
& 100  &  0.0092(0.0102) & 0.0092(0.0102) & 0.0093(0.0103) & 0.0092(0.0102) & 0.0094(0.0102) \\
& 110  &  0.0090(0.0091) & 0.0090(0.0091) & 0.0090(0.0091) & 0.0090(0.0091) & 0.0092(0.0091) \\
& 120  &  0.0079(0.0085) & 0.0079(0.0085) & 0.0080(0.0085) & 0.0079(0.0085) & 0.0080(0.0085) \\
\hline
              & 20  & 0.0509(0.0374) & 0.0374(0.0355) & 0.0359(0.0354) & 0.0251(0.0339) & 0.0114(0.0326) \\
              & 30  & 0.0329(0.0234) & 0.0233(0.0225) & 0.0221(0.0224) & 0.0148(0.0218) & 0.0056(0.0213) \\
              & 40  & 0.0260(0.0170) & 0.0185(0.0165) & 0.0176(0.0164) & 0.0121(0.0161) & 0.0052(0.0158) \\
              & 50  & 0.0215(0.0139) & 0.0153(0.0136) & 0.0147(0.0135) & 0.0102(0.0133) & 0.0047(0.0131) \\
              & 60  & 0.0154(0.0113) & 0.0101(0.0111) & 0.0097(0.0111) & 0.0059(0.0109) & 0.0014(0.0108) \\
$1.33$    & 70  & 0.0171(0.0099) & 0.0125(0.0097) & 0.0121(0.0097) & 0.0090(0.0095) & 0.0049(0.0094) \\
              & 80  & 0.0118(0.0083) & 0.0078(0.0081) & 0.0074(0.0081) & 0.0047(0.0081) & 0.0012(0.0080) \\
              & 90  & 0.0107(0.0073) & 0.0073(0.0072) & 0.0068(0.0072) & 0.0044(0.0071) & 0.0012(0.0071) \\
              & 100 & 0.0090(0.0067) & 0.0058(0.0067) & 0.0054(0.0066) & 0.0032(0.0066) & 0.0005(0.0066) \\
              & 110 & 0.0085(0.0061) & 0.0056(0.0061) & 0.0053(0.0061) & 0.0033(0.0060) & 0.0007(0.0060) \\
              & 120 & 0.0084(0.0055) & 0.0056(0.0055) & 0.0054(0.0055) & 0.0036(0.0054) & 0.0012(0.0054) \\
\hline
             &  20  &  0.0431(0.0290) &  0.0221(0.0272) &  0.0208(0.0271) &  0.0026(0.0261) &  0.0191(0.0261) \\
             &  30  &  0.0353(0.0204) &  0.0205(0.0193) &  0.0199(0.0194) &  0.0077(0.0187) &  0.0065(0.0185) \\
             &  40  &  0.0254(0.0152) &  0.0140(0.0147) &  0.0135(0.0146) &  0.0044(0.0143) &  0.0061(0.0142) \\
             &  50  &  0.0220(0.0121) &  0.0127(0.0117) &  0.0122(0.0117) &  0.0050(0.0115) &  0.0033(0.0114) \\
             &  60  &  0.0157(0.0097) &  0.0078(0.0095) &  0.0075(0.0095) &  0.0015(0.0093) &  0.0054(0.0093) \\
$ 2.27$   &  70  &  0.0141(0.0083) &  0.0073(0.0081) &  0.0070(0.0081) &  0.0019(0.0080) &  0.0040(0.0079) \\
             &  80  &  0.0143(0.0072) &  0.0083(0.0070) &  0.0081(0.0070) &  0.0036(0.0069) &  0.0016(0.0069) \\
             &  90  &  0.0117(0.0065) &  0.0064(0.0064) &  0.0061(0.0064) &  0.0022(0.0063) &  0.0024(0.0063) \\
             &  100 &  0.0110(0.0056) &  0.0062(0.0055) &  0.0060(0.0055) &  0.0024(0.0055) &  0.0017(0.0055) \\
             &  110 &  0.0079(0.0053) &  0.0035(0.0052) &  0.0033(0.0052) &  0.0000(0.0052) &  0.0037(0.0052) \\
             &  120 &  0.0072(0.0048) &  0.0031(0.0047) &  0.0030(0.0047) &  0.0000(0.0047) &  0.0035(0.0047) \\
							\hline
\end{tabular}}
\label{tableres1}
\end{table}

\begin{table*}[!h]
\centering
\caption{The $CP_{95\%}$ from the estimates of $\mu$ and $\Omega$ considering different values of $n$  with $N=10,000,000$ simulated samples, using the estimation methods:  1 - MLE, 2- Jeffreys's rule, 3 -  Reference 1  prior, 4 - Reference  2 prior, and 5 - Tibshirani Prior.}
\begin{tabular}{c|c|c|c|c|c|c}
\hline
 $\boldsymbol{\theta}$  & n &   MLE & Jeffreys & Ref $W$ & Ref $H$ & Tibshirani \\
\hline
\multirow{11}{*}{$H = -0.01$ } 
& 20  & 0.971 & 0.943 & 0.942 & 0.947 & 0.955 \\
& 30  & 0.973 & 0.941 & 0.941 & 0.944 & 0.948 \\
& 40  & 0.979 & 0.949 & 0.950 & 0.952 & 0.953 \\
& 50  & 0.977 & 0.944 & 0.942 & 0.944 & 0.947 \\
& 60  & 0.978 & 0.945 & 0.948 & 0.946 & 0.949 \\
& 70  & 0.980 & 0.948 & 0.949 & 0.949 & 0.950 \\
& 80  & 0.981 & 0.947 & 0.947 & 0.948 & 0.949 \\
& 90  & 0.980 & 0.946 & 0.948 & 0.947 & 0.948 \\
& 100 & 0.982 & 0.945 & 0.947 & 0.948 & 0.948 \\
& 110 & 0.985 & 0.948 & 0.947 & 0.950 & 0.950 \\
& 120 & 0.982 & 0.948 & 0.946 & 0.948 & 0.948 \\
\hline
\multirow{11}{*}{$H = 1.33$ } 
& 20  & 0.923 & 0.932 & 0.936 & 0.943 & 0.951 \\
& 30  & 0.937 & 0.940 & 0.942 & 0.947 & 0.953 \\
& 40  & 0.941 & 0.946 & 0.946 & 0.952 & 0.955 \\
& 50  & 0.936 & 0.942 & 0.943 & 0.945 & 0.950 \\
& 60  & 0.941 & 0.946 & 0.946 & 0.948 & 0.951 \\
& 70  & 0.937 & 0.945 & 0.942 & 0.946 & 0.946 \\
& 80  & 0.943 & 0.948 & 0.948 & 0.949 & 0.950 \\
& 90  & 0.946 & 0.952 & 0.952 & 0.954 & 0.955 \\
& 100 & 0.945 & 0.950 & 0.948 & 0.950 & 0.951 \\
& 110 & 0.940 & 0.945 & 0.945 & 0.946 & 0.948 \\
& 120 & 0.949 & 0.953 & 0.954 & 0.954 & 0.956 \\
\hline
\multirow{11}{*}{$H = 2.27$}
& 20 & 0.946 & 0.951 & 0.952 & 0.956 & 0.957 \\
& 30 & 0.938 & 0.945 & 0.946 & 0.951 & 0.953 \\
& 40 & 0.937 & 0.941 & 0.941 & 0.946 & 0.947 \\
& 50 & 0.939 & 0.945 & 0.946 & 0.948 & 0.949 \\
& 60 & 0.941 & 0.951 & 0.951 & 0.951 & 0.954 \\
& 70 & 0.943 & 0.954 & 0.956 & 0.955 & 0.957 \\
& 80 & 0.944 & 0.956 & 0.957 & 0.959 & 0.959 \\
& 90 & 0.945 & 0.954 & 0.956 & 0.956 & 0.956 \\
& 100 & 0.946 & 0.960 & 0.960 & 0.962 & 0.963 \\
& 110 & 0.943 & 0.956 & 0.956 & 0.956 & 0.957 \\
& 120 & 0.948 & 0.960 & 0.960 & 0.960 & 0.961 \\
\hline
\end{tabular}
\label{tableres2}
\end{table*}
 
Tables \ref{tableres1} and \ref{tableres2} present the Bias, MSEs, and $CP_{95\%}$ for the MLE and Bayesian estimators of the entropy measure $H$. In particular, the results revealed that:

\begin{enumerate}

    \item For all the parameter estimators, the Bias and MSE approach zero for large $n$, which implies asymptotic unbiasedness, i.e., the Bias approaches zero, and the MSE decreases as the number of samples increases.

    \item \textcolor{black}{
In the first scenario, we obtained results when the model reduces to the exponential distribution, and all the estimators behave similarly in terms of estimates. However, with the Bayesian approach, credibility intervals are more precise.}

    \item \textcolor{black}{The Bayesian results yielded better estimates compared to the classical estimators. In fact, \cite{firth1993bias} demonstrated that a bias correction in the MLE is equivalent to adopting Jeffreys' priors for the exponential family of distributions, including the gamma distribution. This should yield results approximately equal to those we obtain for the posterior estimates using the Jeffreys posterior. However, the confidence intervals differ because they are based on asymptotic results and tend to resemble those of the MLE. More importantly, other priors yielded better outcomes, indicating that the Bayesian approach in this instance provided superior results even when compared with corrective methods applied to the standard MLE.}
    
    \item The posterior means using reference priors 1 and 2 were superior to the posterior means using Jeffreys' prior and MLE. However, the posterior mean using reference prior 2 was consistently superior to the posterior mean using reference prior 1. This performance is validated through the coverage probability informed by the CI. Additionally, the coverage probability was high for all the estimators, and the credibility of the interval increases with sample size.
    
    \item For all estimators, the most significant drop in Bias and MSE was observed when the sample size increased from $20$ to $30$.
    
    \item Overall, the results show that the MLE performed worst, given its high bias and MSE. On the other hand, the posterior estimates using the matching prior provided Bayes estimates with smaller Bias and MSE; it was considered the most adequate prior for estimating $H$.
\end{enumerate}

According to the simulation results, the posterior distribution with the associated matching prior leads to the most precise results with the least bias and MSE. The cited prior outperforms other objective priors and ML estimates considered in this study and therefore should be chosen as the most appropriate prior for inference. Besides, the posterior estimates obtained from the matching prior have superior theoretical properties, such as invariance under one-to-one parameter transformations, consistent sampling, and consistency under marginalization. Therefore, we conclude that the posterior estimates derived from the matching prior distribution are more appropriate and superior for making inferences about the gamma distribution's population parameter. To conduct the Bayesian analysis with the proposed Bayes estimator, we have presented a function in R that can be used for this purpose, the details can be seen in Appendix A.

\section{Application}\label{application}

\subsection{Achaemenid dynasty}

The Achaemenid dynasty of the Achaemenid empire was the royal house of the ancient Persians who ruled over Persia kingdom. {\color{black}It is customary that} authority is {\color{black}transferred} to the descendant of the same bloodline {\color{black}after the death of the emperor}. The Persian Empire was built and expanded through military conquest to extend political control to a broader territory. 
The Persian dynasty suffered several reoccurring political conflicts, assassinations, and wars from internal and external entities, which shaped the political institutions over the years.

Conflicts potentially {\color{black}threaten} an emperor's {\color{black}tenure} duration and ascendancy {\color{black}of a descendant}. {\color{black}An emperor that established} a stable government is {\color{black}likely to have a longer tenure and be succeeded by the emperor's chosen descendant compared with emperors that were unable to establish a stable government}. {\color{black}Hence, tenure duration is a measure resourceful for quantifying the instability in the Achaemenid dynasty. We quantified the variation in the tenure duration of the Achaemenid dynasty using the gamma entropy.} 
The more frequently new emperors ascend the throne, the higher the {\color{black}uncertainty in the political situation, which results in a higher} entropy and more {\color{black}likely the government is} unstable. 

\begin{figure}[!h]
	\centering
	\includegraphics[scale=0.46]{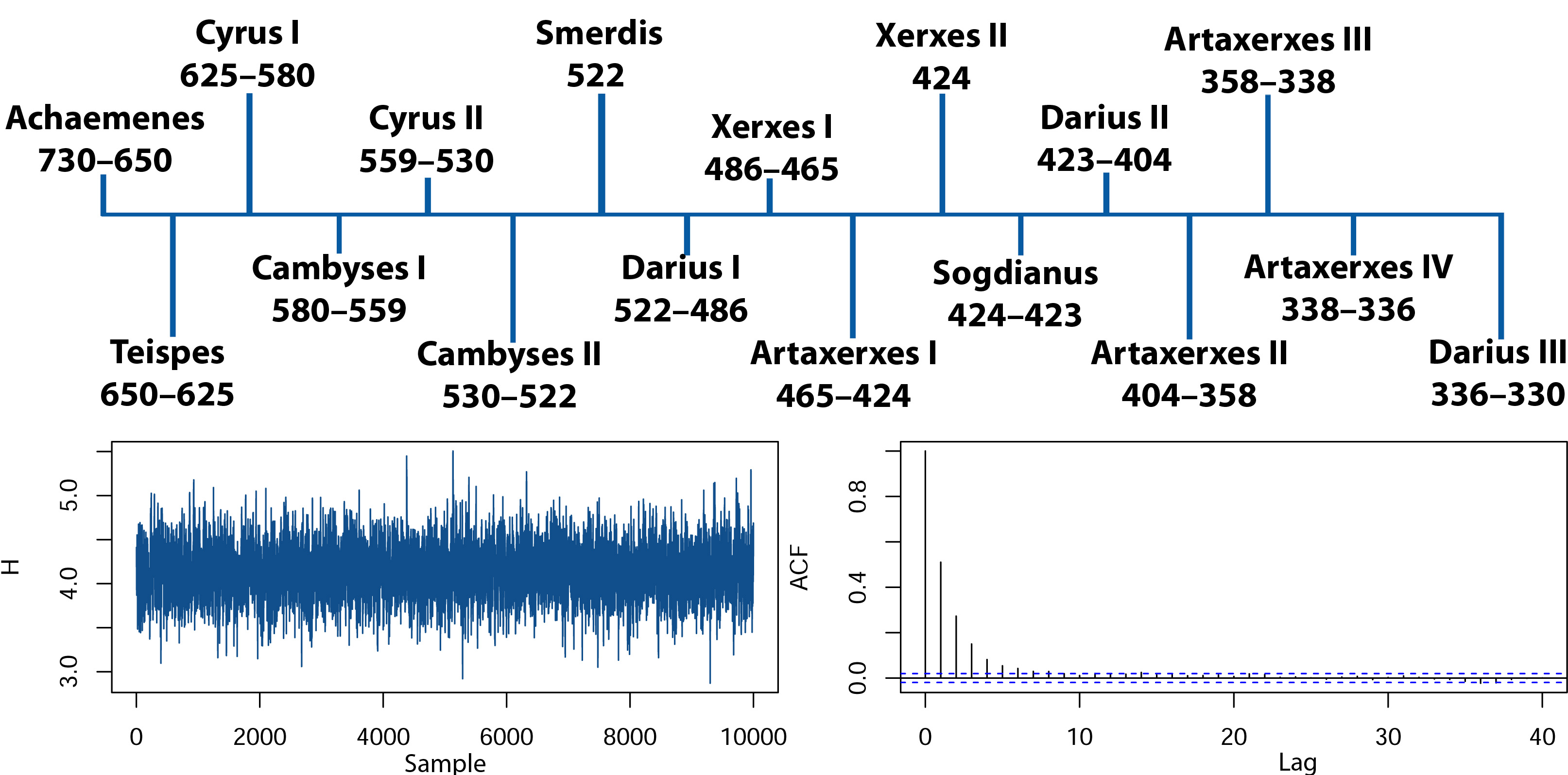}
	\caption{Timeline (BC) containing the data, time series plot of the posterior distribution of the entropy, and autocorrelation plot for the same distribution.}\label{fsimulation1}
\end{figure}

Figure \ref{fsimulation1} (top panel) shows the timeline of the Achaemenid dynasty. {\color{black}The Figure shows the duration of each emperor between emperors Achaemenes and Darius III. From the data, Emperor Achaemenes had the longest tenure of 80 years, and Emperors Smerdis and Xerxes II had the shortest tenures, which were less than a year.} The Kolmogorov-Smirnov (KS) test (statistic D = 0.21) was used to confirm that the data follow a gamma distribution. Figure \ref{fsimulation1} (down-left panel) shows the time series of the Bayesian estimate of the entropy $H$, and the autocorrelation plot (down-right panel). The time series and the autocorrelation plot indicate the chain's convergence, which was also confirmed by the Geweek test \citep{geweke1992evaluating}. For comparison {\color{black}purpose}, the same model was applied to the Roman Empire timeline data, which was {\color{black}previously} analyzed by  \cite{Ramos2020Power}.

Using the posterior distribution obtained from the matching prior, the Bayes estimate of the Achaemenid dynasty's entropy is $4.13$ with a $95\%$ credible interval of $(3.55;4.73)$. Moreover, with the same prior, the posterior estimate for the Roman Empire is $3.08$ with a $95\%$ credible interval of $(2.80;3.36)$. 
The {\color{black}results indicate that} the Achaemenid dynasty had a higher entropy, which implies that the Achaemenid dynasty's political institution was more volatile compared with the Roman Empire. That is, the time between the successive emperors is significantly different, shorter, and irregular for the Achaemenid dynasty, which signifies instability in their political institutions relative to the Roman Empire. These results support the historian's claim that the Achaemenid Empire set out for wars and consequently were exposed to internal and external {\color{black}conflicts}.  

\subsection{Harvest Sugarcane machine}

 Sugarcane farming is pertinent to Brazil's economic growth and has heavily contributed to its Gross Domestic Product (GDP). The production process involves an automated harvesting mechanism, and the interest of the sugarcane farmers is to sustain its harvesting mechanism for an extended period. Moreover, the production chain must be kept in stable conditions to avoid fluctuation in production {\color{black}process and prevent} wastage. We {\color{black}estimated the gamma entropy of the harvesting machine failure times} using the developed Tibshirani prior.  The {\color{black}higher the} entropy, {\color{black}the higher the uncertainty and} severe the irregularities in the production process. Otherwise, the production process is steady. 

The considered data was collected between January 2015 to August 2017, which corresponds to {\color{black}over two years of}
harvests and twenty-one failure times (in days) of the suspension of the sugarcane harvester machine:
11, 19, 36, 4, 8, 11, 39, 74, 168, 27, 116, 3, 34,   1,  46, 12, 2, 56, 14, 52, 14. The KS test was used (statistic D = 0.12) to confirm that the failure times follow a gamma distribution.

Figure \ref{fsimulation1} presents the time series of the Bayesian estimation of the average entropy (left panel) and the correlation lags (right panel). The convergence of the estimate was tested using the Geweek test.
%
\begin{figure}[!h]
	\centering
	\includegraphics[scale=0.42]{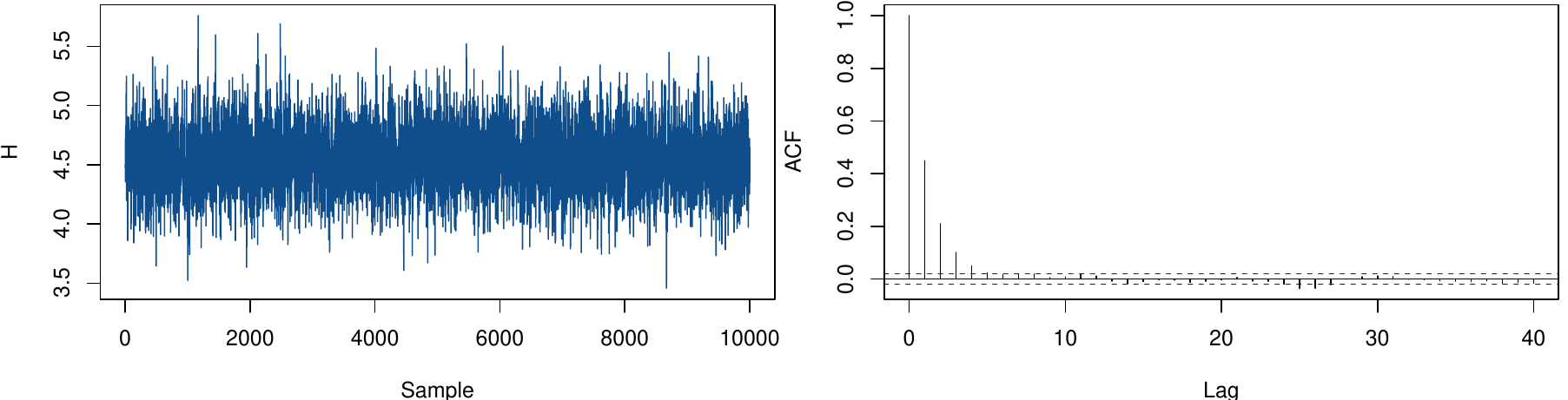}
	\caption{Time series plot of the posterior distribution of the entropy and autocorrelation plot for the same distribution.}\label{fsimulation5}
\end{figure} 
The Bayes estimate of the entropy for the failure times of the harvest sugarcane machine is $4.55$ with a $95\%$ credible interval of $(4.04; 5.09)$. The estimated entropy, {\color{black} in combination with the state of the harvesting matching, can be used for benchmarking future harvesting processes to determine its stability.} 
The harvesting machine must regularly pass a thorough maintenance check within the harvester's life circle to keep a steady production flow. 

\section{Final Remarks}\label{conclusions}

The concept of entropy originated in statistical thermodynamics and was later adapted for application in other fields. In information theory, Shannon entropy measures the uncertainty of a random process. In statistical inference, the parameters of Shannon entropy are determined using the maximum likelihood approach (MLE). However, this approach yields biased results for small samples, and the confidence intervals may not achieve the desired coverage probabilities if the asymptotic assumptions are not met. In this paper, we introduce a fully objective Bayesian analysis to obtain the posterior distribution of Shannon entropy, thereby addressing this limitation.

We considered objective priors, ensuring that the obtained posterior distributions are not overshadowed by prior information. The posterior distributions were derived assuming Jeffreys prior, reference priors, and matching priors, all invariant under one-to-one transformations. Since the obtained priors are improper, they could lead to improper posteriors, which is undesirable. We proved that the obtained posteriors are proper distributions, addressing this issue and enabling the conduct of Bayesian analysis. The posterior mean was considered a Bayes estimator, and given that they may not exist or be finite, we also proved that the posterior means are finite for any sample size. Hence, four posterior distributions were proposed for conducting inference. An intensive simulation study was conducted to select either a Bayesian estimator or the MLE. The posterior distribution using the matching prior yielded better results in terms of bias, mean square error, and coverage probabilities compared to other methods, while the MLE yielded the worst results.

We analyzed a particular case of the gamma distribution, which is a more flexible and general model than the exponential distribution, and has been used to describe many real phenomena. Although we considered a particular case of the gamma distribution, our approach is general and can be extended for any probability distribution function.

The proposed Bayes estimator was implemented in R language, with the code available in the appendix, to estimate the Shannon entropy measure. We applied this implementation to estimate the entropy related to the rule time of the Achaemenid dynasty, which yielded a higher value compared to the Roman Empire. This shows that changes in the throne were less probable in the Achaemenid dynasty than in the Roman Empire, indicating significant instability in their political institutions, which may have contributed to their fall. Furthermore, we analyzed the time until the failure of the suspension in a sugarcane harvesting machine, estimating its entropy using the Bayesian approach.

There are numerous possible extensions to the current work. Other distributions can be considered in the same context, and the Bayes estimator of Shannon entropy can be derived. Different types of entropy measures, such as Hartley, Rényi, and Tsallis entropy, can also be estimated under a Bayesian approach. We plan to explore this line of research in the future.

\section*{Data Availability}

The computer codes and data that support the findings of this study are openly available on GitHub at https://github.com/eosafu/GammaEntropy

\section*{Acknowledgements}

Eduardo Ramos acknowledges financial support from S\~ao Paulo State Research Foundation (FAPESP Proc. 2019/27636-9). Francisco Rodrigues acknowledges financial support from CNPq (grant number 309266/2019-0).  Francisco Louzada is supported by the Brazilian agencies CNPq (grant number 301976/2017-1) and FAPESP (grant number 2013/07375-0).

\section{Appendices}

Here, we provide the proof of the Theorems.

\subsection{Proof of Theorem \ref{theo33}}\label{ccheo33}

\begin{proof} Using the change of variables $\exp(-H)=u\Leftrightarrow du = - \exp(-H)dH$ and denoting $\delta_1(W) = \exp(W + \log(\Gamma(W))+(1-W)\psi(W))$ it follows that 
{\small
\begin{equation*}
\begin{aligned}
d_1(x)&\propto \int_0^\infty \int_{-\infty}^\infty \pi_1(H,W|\boldsymbol{x})\; dH dW \\ &  \propto \int_0^\infty \int_0^\infty\frac{\delta_1(W)^{nW}u^{nW-1}\sqrt{W\psi'(W)-1}}{\Gamma(W)^n}\left\{\prod_{i=1}^n{x_i^W}\right\}\exp\left\{-\delta_1(W)u\sum_{i=1}^n x_i\right\} du dW\\ &
=\int_0^\infty \frac{\delta_1(W)^{nW}\sqrt{W\psi'(W)-1}}{\Gamma(W)^n}\left\{\prod_{i=1}^n{x_i^W}\right\}\int_0^\infty u^{nW-1}\exp\left\{-\delta_1(W)\left(\sum_{i=1}^n x_i\right) u\right\} du dW\\ &
= \int_0^\infty \sqrt{W\psi'(W)-1}\frac{\left\{\prod_{i=1}^n{x_i^W}\right\}}{\left( \sum_{i=1}^n x_i\right)^{nW}} \frac{\Gamma(nW)}{\Gamma(W)^n} dW
=\int_0^1 g_1(W)dW + \int_1^\infty g_1(W) dW,
\end{aligned}
\end{equation*}}
where $g_1(W) =\sqrt{W\psi'(W)-1}\frac{\left\{\prod_{i=1}^n{x_i^W}\right\}}{\left( \sum_{i=1}^n x_i\right)^{nW}} \frac{\Gamma(nW)}{\Gamma(W)^n}>0$ for all $W\in (0,\infty)$. Now, according to \citep{ramos2017bayesian,ramos2018posterior}, we have $\frac{\Gamma(nW)}{\Gamma(W)^n}\underset{W\to 0^+}{\propto}W^{n-1}$ and  $\sqrt{W\psi'(W)-1}\underset{W\to 0^+}{\propto} W^{-1/2}$ and since 
\begin{equation*} \lim_{W\to 0^+} \frac{\left\{\prod_{i=1}^n{x_i^W}\right\}}{\left( \sum_{i=1}^n x_i\right)^{nW}} = 1\Rightarrow \frac{\left\{\prod_{i=1}^n{x_i^W}\right\}}{\left( \sum_{i=1}^n x_i\right)^{nW}}\underset{W\to 0^+}{\propto}1,
\end{equation*}
it follows by Proposition \ref{prop32} that
\begin{equation*} \int_0^1 g_1(W) dW \propto \int_0^1  W^{-1/2}\times 1\times W^{n-1}\; dW < \infty.
\end{equation*}
Moreover, due to \citep{ramos2017bayesian,ramos2018posterior} we have $\frac{\Gamma(nW)}{\Gamma(W)^n}\underset{W\to \infty}{\propto} n^{nW}W^{(n-1)/2}$ and $\sqrt{W\psi'(W)-1}\underset{W\to \infty}{\propto} W^{-1/2}$, and since $x_i$ are not all equal, due to the inequality of the arithmetic and geometric means we have $q=\log\left(\frac{\frac{1}{n} \sum_{i=1}^n x_i}{\sqrt[n]{\prod_{i=1}^n{x_i}}}\right)>0$ and thus it follows that
\begin{equation*}\frac{\left\{\prod_{i=1}^n{x_i^W}\right\}}{\left( \sum_{i=1}^n x_i\right)^{nW}}=\left(\frac{\frac{1}{n} \sum_{i=1}^n x_i}{\sqrt[n]{\prod_{i=1}^n{x_i}}}\right)^{-nW}n^{-nW}= \exp\left(-nqW\right)n^{-nW}.
\end{equation*}
Therefore, from Proposition \ref{prop32} it follows that
\begin{equation*}
\begin{aligned}
\int_1^\infty g_1(W) dW & \propto  \int_1^\infty W^{-1/2}\times \exp\left(-nqW\right)n^{-nW}\times n^{nW}W^{(n-1)/2} dW\\ & = \int_1^\infty W^{n/2-1}\exp\left(-nqW\right) \; dW =\frac{\Gamma(n/2)}{(nq)^{n/2}}<\infty,
\end{aligned}
\end{equation*}
which concludes the proof.
\end{proof}

\subsection{Proof of Theorem \ref{theo34}}\label{ctheo34}

\begin{proof}  Doing the change of variables $\exp(-H)=u\Leftrightarrow du = - \exp(-H)dH$ and denoting $\delta_1(W) = \exp(W + \log(\Gamma(W))+(1-W)\psi(W))$, it follows that 
{\footnotesize

\begin{equation*} 
\begin{aligned}
E_1[H|x]&\propto \int_0^\infty \int_{-\infty}^\infty H \pi_1({\color{blue}W,H}|\boldsymbol{x})\; dH dW \\ &  = \int_0^\infty \int_0^\infty -\log(u)\frac{\delta_1(W)^{nW}u^{nW-1}\sqrt{W\psi'(W)-1}}{\Gamma(W)^n}\left\{\prod_{i=1}^n{x_i^W}\right\}\exp\left\{-\delta_1(W)u\sum_{i=1}^n x_i\right\} du dW\\ &
=\int_0^\infty \frac{\delta_1(W)^{nW}\sqrt{W\psi'(W)-1}}{\Gamma(W)^n}\left\{\prod_{i=1}^n{x_i^W}\right\}\int_0^\infty \left(-\log(u)\right) u^{nW-1}\exp\left\{-\delta_1(W)\left(\sum_{i=1}^n x_i\right) u\right\} du dW.
\end{aligned}
\end{equation*}}
Moreover, from the identity $\psi(z)\Gamma(z)=\Gamma'(z)=\int_0^\infty \log(t)t^{z-1}e^{-t}dz$ one obtains that
\begin{equation*} \int_0^\infty \log(s)s^{z-1}e^{-as}ds = 1/a^z\int_0^\infty \log(t/a)t^{z-1}e^{-t}dt = 1/a^z \left(\psi(z)\Gamma(z)-\log(a)\Gamma(z)\right)
\end{equation*}
and thus, letting $\left|\cdot \right|$ denote the absolute value operator and letting $\delta_2(W) = \left| \psi(nW)\right|+\left|\log(\Gamma(W))\right|+(1+W)|\psi(W)|+W+\left|\log\left(\sum_{i=1}^n x_i\right)\right|$ for all $W>0$, and using the triangle inequality we have
{\small
\begin{equation*}
\begin{aligned}
\left|E_1[H|x]\right| & \propto  \left|  \int_0^\infty  \left(\psi(nW)-\log\left(\delta_1(W)\sum_{i=1}^n x_i\right)\right) \sqrt{W\psi'(W)-1}\frac{\left\{\prod_{i=1}^n{x_i^W}\right\}}{\left( \sum_{i=1}^n x_i\right)^{nW}} \frac{\Gamma(nW)}{\Gamma(W)^n} dW\right|\\
 &\leq   \int_0^\infty \left| \psi(nW)-\log\left(\delta_1(W)\sum_{i=1}^n  x_i\right)\right| \sqrt{W\psi'(W)-1}\frac{\left\{\prod_{i=1}^n{x_i^W}\right\}}{\left( \sum_{i=1}^n x_i\right)^{nW}} \frac{\Gamma(nW)}{\Gamma(W)^n} dW\\
& \leq \int_0^\infty \delta_2(W) \sqrt{W\psi'(W)-1}\frac{\left\{\prod_{i=1}^n{x_i^W}\right\}}{\left( \sum_{i=1}^n x_i\right)^{nW}} \frac{\Gamma(nW)}{\Gamma(W)^n} dW\\ &=\int_0^1 h_1(W) dW + \int_1^\infty h_1(W) dW,\\
\end{aligned}
\end{equation*}}
where $h_1(W)=\delta_2(W) \sqrt{W\psi'(W)-1}\frac{\left\{\prod_{i=1}^n{x_i^W}\right\}}{\left( \sum_{i=1}^n x_i\right)^{nW}} \frac{\Gamma(nW)}{\Gamma(W)^n}$ for all $W>0$.

We shall now prove that $\delta_2(W)\underset{W\to 0^+}{\propto} W^{-1}$ and $\delta_2(W)\underset{W\to \infty}{\propto} W\log(W)$. Indeed, notice that $\delta_2(W)\geq W>0$ for $W>0$. Moreover, since due to \cite{abramowitz} we have $\lim_{W\to 0^+}W\Gamma(W)=1$ and $\lim_{W\to 0^+}-W\psi(W)=1$ it follows that 
\begin{equation*}
\begin{aligned}
&\lim_{W\to 0^+}\frac{\left|\psi(nW)\right|}{W^{-1}}=\lim_{W\to 0^+}\frac{1}{n}\left|(nW)\psi(nW)\right|=\frac{1}{n}\\
&\lim_{W\to 0^+}\frac{ \left|\log\left(\Gamma(W)\right)\right|}{W^{-1}} = \lim_{W\to 0^+} \left|W\log(\Gamma(W))-W\log(W)\right|=\left|0\cdot \log(1)-0\right| = 0\\
& \lim_{W\to 0^+}\frac{ (1+W)\left|\psi(W)\right|}{W^{-1}}=\lim_{W\to 0^+}(1+W)\left|W\psi(W)\right|=1\mbox{ and }\\
&\lim_{W\to 0^+} \frac{W+\left|\log\left(\sum_{i=1}^n x_i\right)\right|}{W^{-1}} =\lim_{W\to 0^+} \left(W^2+W\left|\log\left(\sum_{i=1}^n x_i\right)\right|\right) =0
\end{aligned}
\end{equation*}
and thus
\begin{equation*} \lim_{W \to 0^+}\frac{\delta_2(W)}{W^{-1}} = \frac{1}{n}+1\Rightarrow \delta_2(W)\underset{W\to 0^+}{\propto} \frac{1}{W}.
\end{equation*}
On the other hand, since due to \cite{abramowitz} we have $\lim_{W\to \infty}\frac{\psi(W)}{\log(W)}=1$, it follows from the L'hopital rule that
\begin{equation*}
\lim_{W\to \infty} \frac{\log(\Gamma(W))}{W(\log(W)+1)}=\lim_{W\to \infty} \frac{(\log(\Gamma(W))'}{(W(\log(W)+1))'}= \lim_{W\to \infty} \frac{\psi(W)}{\log(W)}=1,
\end{equation*}
and therefore, considering $W\geq 1$ we have
\begin{equation*}
\begin{aligned}
&\lim_{W\to \infty}\frac{\left|\psi(W)\right|}{W(\log(W)+1)} = \lim_{W\to \infty}\frac{1}{W}\frac{1}{(1+\log(W)^{-1})}\left|\frac{\psi(W)}{\log(W)}\right| = 0,\\
&\lim_{W\to \infty} \frac{\left|\log(\Gamma(W))\right|}{W(\log(W)+1)}=
\lim_{W\to \infty}\left| \frac{\log(\Gamma(W))}{W(\log(W)+1)}\right|=1,\\
&\lim_{W\to \infty} \frac{(1+W)\left| \psi(W)\right|}{W(\log(W)+1)} =\lim_{W\to \infty} \left(1+W^{-1}\right)\frac{1}{\left(1+\log(W)^{-1}\right)}\left|\frac{\psi(W)}{\log(W)}\right|=1,\mbox{ and }\\
&\lim_{W\to \infty} \frac{W+\left|\log\left(\sum_{i=1}^n x_i\right)\right|}{W(\log(W)+1)} = \lim_{W\to \infty} \left(\frac{1}{\log(W)+1}+\frac{\left|\log\left(\sum_{i=1}^n x_i\right)\right|}{W(\log(W)+1)}\right)=0,
\end{aligned}
\end{equation*}
and thus
\begin{equation*} \lim_{W\to \infty}\frac{\delta_2(W)}{W(\log(W)+1)}=2\Rightarrow \delta_2(W)\underset{W\to \infty}{\propto} W\log(W).
\end{equation*}
Therefore, combining the obtained proportionality  $\delta_2(W)\underset{W\to 0^+}{\propto} W^{-1}$ with the proportionalities proved in Theorem \ref{theo33} and using Proposition \ref{prop32} we have
\begin{equation*} \int_0^1 h_1(W) dW \propto \int_0^1 W^{-1}\times W^{-1/2}\times 1\times W^{n-1}\; dW < \infty.
\end{equation*}
Finally, using the proportionality $\delta_2(W)\underset{W\to \infty}{\propto} W\log(W)$, letting $q>0$ be as in the proof of Theorem \ref{theo33} and using that $\log(W)+1\leq \exp(\log(W))=W$ for $W\geq 1$, it follows from the proportionalities proved during Theorem \ref{theo33} and from Proposition \ref{prop32} that
\begin{equation*}
\begin{aligned}
\int_1^\infty h_1(W) dW & \propto  \int_1^\infty W(\log(W)+1)\times  W^{-1/2}\times \exp\left(-nqW\right)n^{-nW}\times n^{nW}W^{(n-1)/2} dW\\ & \leq \int_1^\infty W^{(n/2+2)-1}\exp\left(-nqW\right) \; dW =\frac{\Gamma(n/2+2)}{(nq)^{n/2+2}} <\infty,
\end{aligned}
\end{equation*}
which concludes the proof.
\end{proof}

\subsection{Proof of Theorem \ref{theo35}}\label{ctheo35}

\begin{proof} Doing the change of variables $\exp(-H)=u\Leftrightarrow du = - \exp(-H)dH$, denoting $\delta_1(W) = \exp(W + \log(\Gamma(W))+(1-W)\psi(W))$ and proceeding analogously as in the proof of Theorem \ref{theo33} we have \begin{equation*}
\begin{aligned}
d_2(w)\propto &\int_0^\infty \int_{-\infty}^\infty \pi_2(W,H|\boldsymbol{x})\; dH dW \propto \int_0^1 g_2(W)dW + \int_1^\infty g_2(W) dW,
\end{aligned}
\end{equation*}
where $g_2(W) =\sqrt{\psi'(W)}\frac{\left\{\prod_{i=1}^n{x_i^W}\right\}}{\left( \sum_{i=1}^n x_i\right)^{nW}} \frac{\Gamma(nW)}{\Gamma(W)^n}>0$ for all $W\in (0,\infty)$. Now, according to \citep{ramos2017bayesian,ramos2018posterior}, we have $\frac{\Gamma(nW)}{\Gamma(W)^n}\underset{W\to 0^+}{\propto}W^{n-1}$ and  $\sqrt{\psi'(W)}\underset{W\to 0^+}{\propto} W^{-1}$, and since we proved in Theorem \ref{theo33} that $\frac{\left\{\prod_{i=1}^n{x_i^W}\right\}}{\left( \sum_{i=1}^n x_i\right)^{nW}}\underset{W\to 0^+}{\propto}1$,
it follows from Proposition \ref{prop32} that
\begin{equation*} \int_0^1 g_2(W) dW \propto \int_0^1  W^{-1}\times 1\times W^{n-1}\; dW < \infty.
\end{equation*}
Moreover, from \cite{abramowitz} we have $\sqrt{\psi'(W)}\underset{W\to \infty}{\propto} W^{-1/2}$, which combined with $\sqrt{W\psi'(W)-1}\underset{W\to \infty}{\propto} W^{-1/2} $ implies in $\sqrt{\psi'(W)}\underset{W\to \infty}{\propto} \sqrt{W\psi'(W)-1}$. Therefore it follows that $g_2(W)\underset{W\to \infty}{\propto} g_1(W)$. and by Proposition \ref{prop32} it follows that
\begin{equation*}
\begin{aligned}
\int_1^\infty g_2(W) dW \propto \int_1^\infty g_1(W) \; dW <\infty,
\end{aligned}
\end{equation*}
which concludes the proof.
\end{proof}

\subsection{Proof of Theorem \ref{theo36}}\label{ctheo36}

\begin{proof}  Proceeding analogously as in the proof of Theorem \ref{theo34} it follows that \begin{equation*} 
\begin{aligned}
\left|E_2[H|x]\right|&\propto \int_0^\infty \left|\int_{-\infty}^\infty H \pi_2(H,W|\boldsymbol{x})\; dH dW\right| \\ &  \leq \int_0^\infty \delta_2(W) \sqrt{\psi'(W)}\frac{\left\{\prod_{i=1}^n{x_i^W}\right\}}{\left( \sum_{i=1}^n x_i\right)^{nW}} \frac{\Gamma(nW)}{\Gamma(W)^n} dW\\ &=\int_0^1 h_2(W) dW + \int_1^\infty h_2(W) dW,
\end{aligned}
\end{equation*}
where $\delta_2(W)$ is the same as defined in the proof of Theorem \ref{theo34} and
\begin{equation*}
    \begin{aligned}
h_2(W) =\delta_2(W) \sqrt{\psi'(W)}\frac{\left\{\prod_{i=1}^n{x_i^W}\right\}}{\left( \sum_{i=1}^n x_i\right)^{nW}} \frac{\Gamma(nW)}{\Gamma(W)^n} \cdot
\end{aligned}
\end{equation*}
Since in the proof of Theorem \ref{theo34} we showed that $ \delta_2(W)\underset{W\to 0^+}{\propto} W^{-1}$, together with the proportionalities proved in Theorem \ref{theo33} and Proposition \ref{prop32} we have
\begin{equation*} \int_0^1 h_2(W) dW \propto \int_0^1 W^{-1}\times W^{-1}\times 1\times W^{n-1}\; dW < \infty.
\end{equation*}
Finally, from the proof of Theorem \ref{theo35} we know that
$\sqrt{\psi'(W)}\underset{W\to \infty}{\propto} \sqrt{W\psi'(W)-1}$, which implies directly that $h_2(W)\underset{W\to \infty}{\propto} h_1(W)$, and thus from Proposition \ref{prop32} it follows that
\begin{equation*}
\begin{aligned}
\int_1^\infty h_2(W) dW \propto \int_1^\infty h_1(W) \; dW <\infty,
\end{aligned}
\end{equation*}
which concludes the proof.
\end{proof}

\subsection{Proof of Theorem \ref{theo39}}\label{ctheo39}

\begin{proof}
 Doing the change of variables $\exp(-H)=u\Leftrightarrow du = - \exp(-H)dH$, denoting $\delta_1(W) = \exp(W + \log(\Gamma(W))+(1-W)\psi(W))$ and proceeding analogously as in the proof of Theorem \ref{theo33} we have
 \begin{equation*}
\begin{aligned}
d_4(w)\propto &\int_0^\infty \int_0^\infty \pi_{4}(W,H|\boldsymbol{x})\; dH dW \propto \int_0^1 g_{4}(W)dW + \int_1^\infty g_{4}(W) dW,
\end{aligned}
\end{equation*}
where $g_{4}(W) =\frac{(W\psi'(W)-1)}{\sqrt{W}}\frac{\left\{\prod_{i=1}^n{x_i^W}\right\}}{\left( \sum_{i=1}^n x_i\right)^{nW}} \frac{\Gamma(nW)}{\Gamma(W)^n}>0$ for all $W\in (0,\infty)$. Now, according to \citep{ramos2017bayesian,ramos2018posterior}, we have $\frac{\Gamma(nW)}{\Gamma(W)^n}\underset{W\to 0^+}{\propto}W^{n-1}$ and $\sqrt{W\psi'(W)-1}\underset{W\to 0^+}{\propto} W^{-1/2}$, which implies in particular that $\frac{(W\psi'(W)-1)}{\sqrt{W}}\underset{W\to 0^+}{\propto} W^{-3/2}$, and since we already proved in Theorem \ref{theo33} that $ \frac{\left\{\prod_{i=1}^n{x_i^W}\right\}}{\left( \sum_{i=1}^n x_i\right)^{nW}}\underset{W\to 0^+}{\propto}1$,
it follows by Proposition \ref{prop32} that
\begin{equation*} \int_0^1 g_{4}(W) dW \propto \int_0^1  W^{-3/2}\times 1\times W^{n-1}\; dW < \infty.
\end{equation*}
Moreover, due to \citep{ramos2017bayesian,ramos2018posterior} we have $\frac{\Gamma(nW)}{\Gamma(W)^n}\underset{W\to \infty}{\propto} n^{nW}W^{(n-1)/2}$ and $\sqrt{W\psi'(W)-1}\underset{W\to \infty}{\propto} W^{-1/2}$, which implies in particular that $\frac{W\psi'(W)-1}{\sqrt{W}}\underset{W\to \infty}{\propto} W^{-3/2}$, and since we already proved in Theorem \ref{theo33} that $\frac{\left\{\prod_{i=1}^n{x_i^W}\right\}}{\left( \sum_{i=1}^n x_i\right)^{nW}}=\exp\left(-nqW\right)n^{-nW}$, where $q=\log\left(\frac{\frac{1}{n} \sum_{i=1}^n x_i}{\sqrt[n]{\prod_{i=1}^n{x_i}}}\right)>0$, by Proposition \ref{prop32} it follows that
\begin{equation*}
\begin{aligned}
\int_1^\infty g_{4}(W) dW & \propto \int_1^\infty W^{-3/2}\times \exp\left(-nqW\right)n^{-nW}\times n^{nW}W^{(n-1)/2} dW\\
& = \int_1^\infty W^{(n/2-1)-1}\exp\left(-nqW\right) \; dW =\frac{\Gamma(n/2-1)}{(nq)^{n/2-1}}<\infty,
\end{aligned}
\end{equation*}
which concludes the proof.
\end{proof}

\subsection{Proof of Theorem \ref{theo310}}\label{ctheo310}

\begin{proof}  Proceeding analogously as in the proof of Theorem \ref{theo34} it follows that \begin{equation*} 
\begin{aligned}
\left|E_4[H|x]\right|&\propto \int_0^\infty \left|\int_{-\infty}^\infty H \pi_4(H,W|\boldsymbol{x})\; dH dW\right| \\ &  \leq \int_0^\infty \delta_2(W) (W\psi'(W)-1)\frac{\left\{\prod_{i=1}^n{x_i^W}\right\}}{\left( \sum_{i=1}^n x_i\right)^{nW}} \frac{\Gamma(nW)}{\Gamma(W)^n} dW \\ & =\int_0^1 g_4(W) dW + \int_1^\infty g_4(W) dW,
\end{aligned}
\end{equation*}
where $\delta_2(W)$ is given as in the proof of Theorem \ref{theo34} and
\begin{equation*}
    \begin{aligned}
    h_4(W)=\delta_2(W) (W\psi'(W)-1)\frac{\left\{\prod_{i=1}^n{x_i^W}\right\}}{\left( \sum_{i=1}^n x_i\right)^{nW}} \frac{\Gamma(nW)}{\Gamma(W)^n}.
\end{aligned}
\end{equation*}
Since in the proof of Theorem \ref{theo34} we showed that $ \delta_2(W)\underset{W\to 0^+}{\propto} W^{-1}$, together with the proportionalities proved in Theorem \ref{theo33} and Proposition \ref{prop32} we have
\begin{equation*} \int_0^1 h_4(W) dW \propto \int_0^1 W^{-1}\times W^{-3/2}\times 1\times W^{n-1}\; dW < \infty.
\end{equation*}
Moreover, letting $q>0$ as in the proof of Theorem \ref{theo33}, since we proved during the proof of Theorem \ref{theo34} that $\delta_2(W)\underset{W\to \infty}{\propto} W(\log(W)+1)$ and since $\log(W)+1\leq \exp(\log(W))=W$ for $W\geq 1$ it follows from Proposition \ref{prop32} that
\begin{equation*}
\begin{aligned}
\int_1^\infty h_4(W) dW & \propto  \int_1^\infty W(\log(W)+1)\times  W^{-3/2}\times \exp\left(-nqW\right)n^{-nW}\times n^{nW}W^{(n-1)/2} dW\\ & \leq \int_1^\infty W^{(n/2+1)-1}\exp\left(-nqW\right) \; dW =\frac{\Gamma(n/2+1)}{(nq)^{n/2+1}}<\infty,
\end{aligned}
\end{equation*}
which concludes the proof.
\end{proof}

\end{document}